\documentclass[11pt, twoside, leqno]{article}

\usepackage{amssymb}
\usepackage{amsmath}
\usepackage{amsthm}
\usepackage{color}
\usepackage{mathrsfs}
\usepackage{txfonts}

\usepackage{indentfirst}

\allowdisplaybreaks

\def\rr{{\mathbb R}}
\def\rn{{\mathbb{R}^n}}
\def\zz{{\mathbb Z}}
\def\cc{{\mathbb C}}
\def\nn{{\mathbb N}}
\def\cp{{\mathcal P}}
\def\cs{{\mathcal S}}

\def\fz{\infty }
\def\az{\alpha}
\def\bz{\beta}
\def\dz{\delta}
\def\gz{\gamma}
\def\lz{\lambda}

\def\lf{\left}
\def\r{\right}

\def\ls{\lesssim}
\def\noz{\nonumber}
\def\wz{\widetilde}
\def\com{\complement}

\def\supp{\mathop\mathrm{\,supp\,}}
\def\esup{\mathop\mathrm{\,ess\,sup\,}}
\def\B{\mathfrak{B}}
\def\var{\varepsilon}
\def\HL{M_{{\rm HL}}}

\def\vh{{H_A^{p(\cdot)}(\rn)}}
\def\vhlpq{{H_A^{p(\cdot),q}(\rn)}}
\def\vah{{H_A^{p(\cdot),r,s}(\rn)}}
\def\vmh{{H_A^{p(\cdot),r,s,\varepsilon}(\rn)}}
\def\vfah{{H_{A,\,{\rm fin}}^{p(\cdot),r,s}(\rn)}}
\def\vfahfz{{H_{A,\,{\rm fin}}^{p(\cdot),\fz,s}(\rn)}}
\def\lfz{{L^{\fz}(\rn)}}
\def\lv{{L^{p(\cdot)}(\rn)}}
\def\hvz{H_A^{\varphi}(\rn)}

\newtheorem{theorem}{Theorem}[section]
\newtheorem{lemma}[theorem]{Lemma}

\theoremstyle{definition}
\newtheorem{remark}[theorem]{Remark}
\newtheorem{definition}[theorem]{Definition}
\renewcommand{\appendix}{\par
   \setcounter{section}{0}%
   \setcounter{subsection}{0}%
   \setcounter{subsubsection}{0}%
   \gdef\thesection{\@Alph\c@section}%
   \gdef\thesubsection{\@Alph\c@section.\@arabic\c@subsection}%
   \gdef\theHsection{\@Alph\c@section.}%
   \gdef\theHsubsection{\@Alph\c@section.\@arabic\c@subsection}%
   \csname appendixmore\endcsname
 }

\numberwithin{equation}{section}

\begin{document}

\arraycolsep=1pt

\title{\bf\Large Molecular Characterizations of Variable Anisotropic Hardy Spaces
with Applications to Boundedness of Calder\'on--Zygmund Operators
\footnotetext{\hspace{-0.35cm} 2010 {\it
Mathematics Subject Classification}. Primary 42B35;
Secondary 42B30, 42B20, 46E30.
\endgraf {\it Key words and phrases.}
expansive matrix, (variable) Hardy space, molecule, Calder\'{o}n--Zygmund operator.
\endgraf The author is supported by the Fundamental Research Funds for the Central Universities
(Grant No.~JSX200211).}}
\author{Jun Liu}
\date{}
\maketitle

\vspace{-0.8cm}

\begin{center}
\begin{minipage}{13cm}
{\small {\bf Abstract}\quad
Let $p(\cdot):\ \mathbb{R}^n\to(0,\infty]$ be a variable exponent
function satisfying the globally log-H\"{o}lder continuous condition
and $A$ a general expansive matrix on $\mathbb{R}^n$.
Let $H_A^{p(\cdot)}(\mathbb{R}^n)$ be the variable anisotropic
Hardy space associated with $A$ defined via the non-tangential grand maximal function.
In this article, via the known atomic characterization of $H_A^{p(\cdot)}(\mathbb{R}^n)$,
the author establishes its molecular characterization with the known best possible
decay of molecules.
As an application, the author obtains a criterion on the boundedness of
linear operators on $H_A^{p(\cdot)}(\mathbb{R}^n)$, which is used to prove
the boundedness of anisotropic Calder\'on--Zygmund operators on
$H_A^{p(\cdot)}(\mathbb{R}^n)$. In addition, the boundedness of
anisotropic Calder\'on--Zygmund operators
from $H_A^{p(\cdot)}(\mathbb{R}^n)$ to the variable Lebesgue space
$L^{p(\cdot)}(\mathbb{R}^n)$ is also presented. All these
results are new even in the classical isotropic setting.}
\end{minipage}
\end{center}

\vspace{0.2cm}

\section{Introduction\label{s1}}

The main purpose of this article is to establish a molecular characterization of the
variable anisotropic Hardy space $\vh$ from \cite{lwyy18}, where $p(\cdot):\ \rn\to(0,\fz]$
is a variable exponent function satisfying the so-called globally log-H\"{o}lder continuous
condition [see \eqref{2e4} and \eqref{2e5} below] and $A$ a general expansive matrix on $\rn$
(see Definition \ref{2d1} below). It is well known that a molecule is a natural
generalization of a atom with the support condition replaced by some decay condition.
Note that, to obtain the boundedness of linear operators, which
are bounded on $L^2(\rn)$, from the classic Hardy space $H^p(\rn)$
to the Lebesgue space $L^p(\rn)$ with $p\in(0,1]$, it suffices to show that the
$L^p(\rn)$ norm of the image of these operators acting on $(p,2,s)$-atoms can be uniformly
controlled by a harmless positive constant (see \cite[Corollary 1.3]{hll12}).
However, it is complicated to investigate the boundedness of these
operators on $H^p(\rn)$ due to the fact that the image of these operators acting on atoms
may no longer be atoms. Fortunately, instead of atoms,
one can use molecules because many of these linear
operators (for instance, one of the most basic operators in harmonic analysis,
Calder\'on--Zygmund operators) usually map an atom into a harmless positive constant multiple
of a related molecule, which further implies the desired
boundedness. Thus, the molecular characterization plays a key role in studying the boundedness
of many important operators on Hardy-type spaces;
see, for instance, \cite{cgn19,cw77,ll02,lhy18,lql,lu,tw80}.

Recall that, as a generalization of the classical Hardy space $H^p(\rn)$, the variable Hardy
space $H^{p(\cdot)}(\rn)$, with the constant exponent $p$ replaced by a variable
exponent function $p(\cdot):\ \rn\to(0,\fz]$, was first introduced by Nakai and Sawano
\cite{ns12} and, independently, by Cruz-Uribe and Wang \cite{cw14} with some weaker assumptions
on $p(\cdot)$ than those used in \cite{ns12}. Later, Sawano \cite{s10}, Yang et al. \cite{yzn16}
and Zhuo et al. \cite{zyy18} further completed the real-variable
theory of these variable Hardy spaces.
For more progress about function spaces with variable exponents, we refer the reader to
\cite{ah10,abr16,cf13,dhhr11,dhr09,jzz17,kv14,xu08,yyyz16,zsy16}. On the other hand, motivated
by the important role of discrete groups of dilations in wavelet theory, Bownik \cite{mb03}
originally introduced the anisotropic Hardy space $H_A^p(\rn)$, with $p\in(0,\fz)$,
which also gave a unified framework of the real-variable theory of both the classical
Hardy space and the parabolic Hardy space of Calder\'{o}n and Torchinsky \cite{ct75}.
Later on, Bownik et al. \cite{blyz08} further extended the anisotropic Hardy space
to the weighted setting. Nowadays, the anisotropic setting has proved useful not only
in developing function spaces, but also in many other branches such as the wavelet theory
(see, for instance, \cite{bb12,mb03,dhp09}) and partial differential equations
(see, for instance, \cite{bw19,jm06}).

Recently, Liu et al. \cite{lwyy18} introduced the variable anisotropic Hardy space
$\vh$ associated with $A$, via the non-tangential grand maximal function, and established
its various real-variable characterizations, respectively, by means of the radial or
the non-tangential maximal functions, atoms, finite atoms,
the Lusin area function, the Littlewood--Paley $g$-function or $g_{\lz}^\ast$-function.
As applications, the boundedness of the maximal operators of the Bochner--Riesz
and the Weierstrass means from $\vh$ to $\lv$ was also obtained in \cite{lwyy18}.
Moreover, these real-variable characterizations of the space $\vh$ have proved very useful
in the study on the real interpolation between $\vh$ and $L^\fz(\rn)$ (see \cite{lyy17hl}).

Nevertheless, the molecular characterization of $\vh$, which can be conveniently
used to prove the boundedness of many important operators (for instance,
Calder\'on--Zygmund operators) on the space $\vh$, is still missing. Therefore, to
further complete the real-variable theory of variable anisotropic Hardy spaces $\vh$,
in this article, we establish a molecular characterization of the space $\vh$,
where the range of the used decay index $\var$ is the known best possible in some sense
[see Remark \ref{3r1}(iv) below]. As an application, we then obtain a criterion
on the boundedness of linear operators on $\vh$ (see Theorem \ref{4t1} below),
which further implies the boundedness of anisotropic Calder\'on--Zygmund
operators on $\vh$. Finally, the boundedness
of anisotropic Calder\'on--Zygmund operators from $\vh$ to the variable Lebesgue space
$\lv$ is also presented.

The organization of the remainder of this article is as follows.

In Section \ref{s2}, we first give some notation used throughout this article,
and then recall some notions on expansive matrices and homogeneous quasi-norms as well
as variable Lebesgue spaces and variable anisotropic Hardy spaces.

Section \ref{s3} is devoted to establishing a molecular characterization of $\vh$
(see Theorem \ref{3t1} below). To this end, we first introduce the variable anisotropic
molecular Hardy space $\vmh$ (see Definition \ref{3d2} below). Then, by the known atomic
characterization of $\vh$ obtained in \cite[Theorem 4.8]{lwyy18}
(see also Lemma \ref{3l3} below), we easily find that $\vh\subset\vmh$ and the inclusion
is continuous. Thus, to prove Theorem \ref{3t1}, it suffices to show that $\vmh$ is
continuously embedded into $\vh$. Note that, to show the embedding of this type,
the well-known strategy is to decompose a molecule into
an infinite linear combination of the related atoms
(see, for instance, \cite[(7.4)]{lu} or \cite[(3.23)]{llll}),
which does not work in the present situation because the uniformly upper bound
estimate of the dual-bases of the natural projection of each molecule
on the infinite annuli of a dilated ball
(see \cite[(7.2)]{lu} or \cite[(3.18)]{llll}) is still unknown due to
its anisotropic structure. To overcome this difficulty,
by borrowing some ideas from the
proofs of \cite[Theorem 3.12]{lhy18} and \cite[Theorem 5.2]{ns12} as
well as fully using the integral
size condition of a molecule [see Definition \ref{3d1}(i) below], we directly
estimate the non-tangential maximal function of a molecule on the infinite annuli of a
dilated ball [see \eqref{3e7} below], and then obtain that $\vmh\subset\vh$ with
continuous inclusion, which completes the proof of Theorem \ref{3t1}.
Here we should point out that, in the proof of \cite[Theorem 3.12]{lhy18}, the useful
properties of the growth function $\varphi$ play a key role (see \cite[(3.5)]{lhy18});
however, that approach is obviously invalid in the present situation due to its
variable exponent setting. Instead, we use a technical lemma, which can reduce
some estimates related to
$\lv$ norms for some series of functions into dealing with the $L^q(\rn)$ norms of
the corresponding functions [see Lemma \ref{3l2} and \eqref{3e5} below].
In addition, in \cite[Theorem 5.2]{ns12}, Nakai and Sawano established a molecular
characterization of the variable isotropic Hardy space $H^{p(\cdot)}(\rn)$; however,
the molecule used in \cite{ns12} is associated with a particular pointwise size
condition, which is much stronger than the integral size condition of a molecule
used in Theorem \ref{3t1} below. In this sense, the conclusion
of Theorem \ref{3t1} also improves the corresponding one of \cite[Theorem 5.2]{ns12}.

As applications, in Section \ref{s4}, we obtain the boundedness of anisotropic
Calder\'on--Zygmund operators from $\vh$ to itself (see Theorem \ref{4t2} below)
or to the variable Lebesgue space $\lv$ (see Theorem \ref{4t3} below). To this end,
by the known finite atomic characterization of $\vh$ and
the molecular characterization of $\vh$ presented in Section \ref{s3}, we first
establish a useful criterion on the boundedness of linear operators on $\vh$
(see Theorem \ref{4t1} below), which shows that, for any given linear operator $T$,
if it maps each atom into a related molecule, then $T$ has a unique
bounded linear extension on $\vh$. Applying this criterion and an auxiliary lemma
from \cite{lql} (see also Lemma \ref{4l4} below), we then prove Theorem \ref{4t2}.
Moreover, a procedure similar to that used in the proof of Theorem \ref{4t2} with
some technical modifications shows that Theorem \ref{4t3} also holds true.

Finally, we make some conventions on notation.
We always let $\nn:=\{1,2,\ldots\}$, $\zz_+:=\{0\}\cup\nn$
and $\vec0_n$ be the \emph{origin} of $\rn$. For any given multi-index
$\gz:=(\gz_1,\ldots,\gz_n)\in(\zz_+)^n=:\zz_+^n$,
let $|\gz|:=\gz_1+\cdots+\gz_n$ and
$\partial^{\gz}
:=(\frac{\partial}{\partial x_1})^{\gz_1}\cdots(\frac{\partial}{\partial x_n})^{\gz_n}.$
We denote by $C$ a positive constant
which is independent of the main parameters, but may vary from line to line.
The \emph{symbol} $f\ls g$ means $f\le Cg$ and,
if $f\ls g\ls f$, then we write $f\sim g$. If $f\le Cg$ and $g=h$ or $g\le h$,
we then write $f\ls g\sim h$ or $f\ls g\ls h$, \emph{rather than} $f\ls g=h$
or $f\ls g\le h$. In addition,
for any set $E\subset\rn$, we denote by $\mathbf{1}_E$ its \emph{characteristic function},
by $E^\complement$ the
set $\rn\setminus E$ and by $|E|$ its \emph{n-dimensional Lebesgue measure}.
For any $r\in[1,\fz]$, we denote by $r'$
its \emph{conjugate index}, namely, $1/r+1/r'=1$ and by $\lfloor t\rfloor$
the \emph{largest integer not greater than $t$} for any $t\in\mathbb{R}$.

\section{Preliminaries \label{s2}}

In this section, we recall the notions of expansive matrices
and variable anisotropic Hardy spaces (see, for instance, \cite{mb03,lwyy18}).

The following definition of expansive matrices is from \cite{mb03}.

\begin{definition}\label{2d1}
A real $n\times n$ matrix $A$ is called an \emph{expansive matrix}
(shortly, a \emph{dilation}) if
$$\min_{\lz\in\sigma(A)}|\lz|>1,$$
here and thereafter, $\sigma(A)$ denotes the \emph{set of all eigenvalues of $A$}.
\end{definition}

Let $b:=|\det A|$. Then, by \cite[p.\,6, (2.7)]{mb03}, we know
that $b\in(1,\fz)$. From the fact that there exist an open
ellipsoid $\Delta$, with $|\Delta|=1$, and $r\in(1,\infty)$ such that
$\Delta\subset r\Delta\subset A\Delta$ (see \cite[p.\,5, Lemma 2.2]{mb03}),
it follows that, for any $i\in\zz$, $B_i:=A^i\Delta$ is open,
$B_i\subset rB_i\subset B_{i+1}$ and $|B_i|=b^i$.
For any $x\in\rn$ and $i\in\zz$, an ellipsoid $x+B_i$
is called a \emph{dilated ball}. Throughout this article,
we always use $\mathfrak{B}$ to denote the set of all such
dilated balls, namely,
\begin{align}\label{2e1}
\mathfrak{B}:=\lf\{x+B_i:\ x\in\rn,\ i\in\zz\r\}
\end{align}
and let
\begin{align}\label{2e2}
\tau:=\inf\lf\{k\in\zz:\ r^k\ge2\r\}.
\end{align}

The following definition on homogeneous quasi-norms
is just \cite[p.\,6, Definition 2.3]{mb03}.

\begin{definition}\label{2d2}
Let $A$ be a given dilation. A measurable mapping $\rho:\ \rn \to [0,\infty)$
is called a \emph{homogeneous quasi-norm}, associated with $A$, if
\begin{enumerate}
\item[\rm{(i)}] $x\neq\vec0_n$ implies $\rho(x)\in(0,\fz)$;

\item[\rm{(ii)}] for each $x\in\rn$, $\rho(Ax)=b\rho(x)$;

\item[\rm{(iii)}] there exists a constant $C\in[1,\fz)$
such that, for any $x$, $y\in\rn$, $\rho(x+y)\le C[\rho(x)+\rho(y)]$.
\end{enumerate}
\end{definition}

For any given dilation $A$, by \cite[p.\,6, Lemma 2.4]{mb03}, we can use the
\emph{step homogeneous quasi-norm} $\rho$ defined by setting, for any $x\in\rn$,
\begin{equation*}
\rho(x):=\sum_{i\in\zz} b^i\mathbf{1}_{B_{i+1}\setminus B_i}(x)\quad
{\rm when}\ x\neq\vec0_n,\qquad {\rm or\ else}
\quad\rho(\vec0_n):=0
\end{equation*}
for convenience.

For any measurable function $p(\cdot):\ \rn\to(0,\fz]$, let
\begin{align}\label{2e3}
p_-:=\mathop\mathrm{ess\,inf}_{x\in \rn}p(x),\hspace{0.35cm}
p_+:=\mathop\mathrm{ess\,sup}_{x\in \rn}p(x)\hspace{0.35cm}
{\rm and}\hspace{0.35cm}\underline{p}:=\min\{p_-,1\}.
\end{align}
Denote by $\cp(\rn)$ the \emph{set of all measurable functions}
$p(\cdot)$ satisfying $0<p_-\le p_+<\fz$.

Let $f$ be a measurable function on $\rn$ and $p(\cdot)\in\cp(\rn)$.
Then the \emph{modular functional} (or, for simplicity, the \emph{modular})
$\varrho_{p(\cdot)}$, associated with $p(\cdot)$, is defined by setting
$$\varrho_{p(\cdot)}(f):=\int_\rn|f(x)|^{p(x)}\,dx,$$ and the
\emph{Luxemburg} (also called \emph{Luxemburg--Nakano})
\emph{quasi-norm} $\|f\|_{\lv}$ by setting
\begin{equation*}
\|f\|_{\lv}:=\inf\lf\{\lz\in(0,\fz):\ \varrho_{p(\cdot)}(f/\lz)\le1\r\}.
\end{equation*}
Moreover, the \emph{variable Lebesgue space} $\lv$ is defined to be the
set of all measurable functions $f$ satisfying that $\varrho_{p(\cdot)}(f)<\fz$,
equipped with the quasi-norm $\|f\|_{\lv}$.

A function $p(\cdot)\in\cp(\rn)$ is said to satisfy the
\emph{globally log-H\"older continuous condition}, denoted by $p(\cdot)\in C^{\log}(\rn)$,
if there exist two positive constants $C_{\log}(p)$ and $C_\fz$, and
a constant $p_\fz\in\rr$ such that, for any $x,\ y\in\rn$,
\begin{equation}\label{2e4}
|p(x)-p(y)|\le \frac{C_{\log}(p)}{\log(e+1/\rho(x-y))}
\end{equation}
and
\begin{equation}\label{2e5}
|p(x)-p_\fz|\le \frac{C_\fz}{\log(e+\rho(x))}.
\end{equation}

Recall also that a \emph{Schwartz function} is an infinitely differentiable
function $\phi$ satisfying, for any $\ell\in\zz_+$ and multi-index $\az\in\zz_+^n$,
$$\|\phi\|_{\alpha,\ell}:=
\sup_{x\in\rn}[\rho(x)]^\ell\lf|\partial^\alpha\phi(x)\r|<\infty.$$
Let $\cs(\rn)$ be the set of all Schwartz functions as above, equipped with the topology
determined by $\{\|\cdot\|_{\alpha,\ell}\}_{\az\in\zz_+^n,\,\ell\in\zz_+}$, and
$\cs'(\rn)$ its \emph{dual space}, equipped with the weak-$\ast$
topology. For any $N\in\zz_+$, let
\begin{align*}
\cs_N(\rn):=
\lf\{\phi\in\cs(\rn):\ \|\phi\|_{\cs_N(\rn)}:=\sup_{\az\in\zz_+^n,\,|\alpha|\le N}
\sup_{x\in\rn}\lf[\lf|\partial^\alpha\phi(x)\r|\max\lf\{1,\lf[\rho(x)\r]^N\r\}\r]\le1\r\}.
\end{align*}
Throughout this article, for any $\phi\in\cs(\rn)$ and $i\in\zz$, let
$\phi_i(\cdot):=b^{-i}\phi(A^{-i}\cdot)$.

Let $\lz_-$, $\lz_+\in(1,\fz)$ be two numbers such that
$$\lz_-\le\min\{|\lz|:\ \lz\in\sigma(A)\}
\le\max\{|\lz|:\ \lz\in\sigma(A)\}\le\lz_+.$$
In particular, when $A$ is diagonalizable over $\mathbb{C}$, we can let
$$\lz_-:=\min\{|\lz|:\ \lz\in\sigma(A)\}
\quad{\rm and}\quad
\lz_+:=\max\{|\lz|:\ \lz\in\sigma(A)\}.$$
Otherwise, we can choose them sufficiently close to these equalities
in accordance with what we need in our arguments.

\begin{definition}\label{2d3}
Let $\phi\in\cs(\rn)$ and $f\in\cs'(\rn)$. The
\emph{non-tangential maximal function} $M_\phi(f)$,
associated to $\phi$, is defined by setting, for any $x\in\rn$,
\begin{align*}
M_\phi(f)(x):= \sup_{y\in x+B_i,\,i\in\zz}|f\ast\phi_i(y)|.
\end{align*}
Moreover, for any given $N\in\mathbb{N}$, the
\emph{non-tangential grand maximal function} $M_N(f)$ of
$f\in\cs'(\rn)$ is defined by setting, for any $x\in\rn$,
\begin{equation*}
M_N(f)(x):=\sup_{\phi\in\cs_N(\rn)}M_\phi(f)(x).
\end{equation*}
\end{definition}

The following variable anisotropic Hardy spaces were
first introduced by Liu et al. in \cite[Definition 2.4]{lwyy18}.

\begin{definition}\label{2d4}
Let $p(\cdot)\in C^{\log}(\rn)$ and
$N\in\mathbb{N}\cap[\lfloor(\frac1{\underline{p}}-1)\frac{\ln b}{\ln
\lz_-}\rfloor+2,\fz)$,
where $\underline{p}$ is as in \eqref{2e3}.
The \emph{variable anisotropic Hardy space} $\vh$ is defined by setting
\begin{equation*}
\vh:=\lf\{f\in\cs'(\rn):\ M_N(f)\in\lv\r\}
\end{equation*}
and, for any $f\in\vh$, let
$\|f\|_{\vh}:=\| M_N(f)\|_{\lv}$.
\end{definition}

Observe that, in \cite[Theorem 3.10]{lwyy18}, it was proved that
the space $\vh$ is independent of the choice of $N$ as in Definition \ref{2d4}.

\section{Molecular characterization of $\vh$\label{s3}}

In this section, we characterize $\vh$ by means of molecules.
Recall that, for any $q\in(0,\fz]$ and measurable set $E\subset\rn$, the Lebesgue
space $L^q(E)$ is defined to be the set of all measurable functions $f$ on $E$
such that, when $q\in(0,\fz)$,
$$\|f\|_{L^q(E)}:=\lf[\int_E|f(x)|^q\,dx\r]^{1/q}<\fz$$
and
$$\|f\|_{L^\fz(E)}:=\esup_{x\in E}|f(x)|<\fz.$$

The following definition of anisotropic $(p(\cdot),r,s,\var)$-molecules
is from \cite{lql}.

\begin{definition}\label{3d1}
Let $p(\cdot)\in\cp(\rn)$, $r\in(1,\fz]$,
\begin{align}\label{3e2}
s\in\lf[\lf\lfloor\lf(\dfrac1{p_-}-1\r)
\dfrac{\ln b}{\ln\lz_-}\r\rfloor,\fz\r)\cap\zz_+
\end{align}
and $\var\in(0,\fz)$, where $p_-$ is as in \eqref{2e3}. A measurable function $m$
is called an \emph{anisotropic $(p(\cdot),r,s,\var)$-molecule}, associated to some
dilated ball $B:=x_0+B_{i_0}\in\B$ with $x_0\in\rn$, $i_0\in\zz$ and $\B$ as in
\eqref{2e1}, if
\begin{enumerate}
\item[(i)] for each $j\in\zz_+$,
$\|m\|_{L^r(U_j(B))}\le b^{-j\var}|B|^{1/r}\|\mathbf{1}_{B}\|_{\lv}^{-1}$,
where $U_0(B):=B$ and, for any $j\in\nn$,
$U_j(B)=U_j(x_0+B_{i_0}):=x_0+(A^j B_{i_0})\setminus(A^{j-1}B_{i_0})$;

\item[(ii)] for any multi-index $\az\in\zz_+^n$ with $|\az|\le s$,
$\int_{\rn}m(x)x^{\az}\,dx=0$.
\end{enumerate}
\end{definition}

In what follows, for convenience,
we always call an anisotropic $(p(\cdot),r,s,\var)$-molecule
simply by a $(p(\cdot),r,s,\var)$-molecule.
Now, using $(p(\cdot),r,s,\var)$-molecules, we introduce the variable
anisotropic molecular Hardy space $\vmh$ as follows.

\begin{definition}\label{3d2}
Let $p(\cdot)\in C^{\log}(\rn)$, $r$, $s$ and $\var$ be as in Definition \ref{3d1}.
The \emph{variable anisotropic molecular Hardy space} $\vmh$ is defined to be the
set of all $f\in\cs'(\rn)$ satisfying that there exist a sequence
$\{\lz_i\}_{i\in\nn}\subset\cc$
and a sequence of $(p(\cdot),r,s,\var)$-molecules, $\{m_i\}_{i\in\nn}$,
associated, respectively, to $\{B^{(i)}\}_{i\in\nn}\subset\B$ such that
\begin{align*}
f=\sum_{i\in\nn}\lz_im_i\quad\mathrm{in}\quad\cs'(\rn).
\end{align*}
Moreover, for any $f\in\vmh$, let
\begin{align*}
\|f\|_{\vmh}:=
{\inf}\lf\|\lf\{\sum_{i\in\nn}
\lf[\frac{|\lz_i|\mathbf{1}_{B^{(i)}}}{\|\mathbf{1}_{B^{(i)}}\|_{\lv}}\r]^
{\underline{p}}\r\}^{1/\underline{p}}\r\|_{\lv},
\end{align*}
where the infimum is taken over all decompositions of $f$ as above
and $\underline{p}$ as in \eqref{2e3}.
\end{definition}

To establish the molecular characterization of $\vh$, we need several technical
lemmas. First, the following Lemma \ref{3l1} is just \cite[Remark 4.4(i)]{lyy17hl}.

\begin{lemma}\label{3l1}
Let $p(\cdot)\in C^{\log}(\rn)$, $r\in(0,\underline{p})$ and $i\in\zz_+$.
Then there exists a positive
constant $C$ such that, for any sequence $\{B^{(k)}\}_{k\in\nn}\subset\B$,
$$\lf\|\sum_{k\in\nn}\mathbf{1}_{A^i B^{(k)}}\r\|_{\lv}\le
C b^{i/r}\lf\|\sum_{k\in\nn}\mathbf{1}_{B^{(k)}}\r\|_{\lv}.$$
\end{lemma}

By Lemma \ref{3l1} and an argument similar to that used in the proof of
\cite[Lemma 3.15]{hlyy19} with some slight modifications, we obtain the
following useful conclusion; the details are omitted.

\begin{lemma}\label{3l2}
Let $r(\cdot)\in C^{\log}(\rn)$ and $q\in[1,\fz]\cap(r_+,\fz]$
with $r_+$ as in \eqref{2e3}.
Assume that
$\{\lz_i\}_{i\in\nn}\subset\cc$, $\{B^{(i)}\}_{i\in\nn}\subset\B$
and $\{a_i\}_{i\in\nn}\subset L^q(\rn)$ satisfy, for any $i\in\nn$,
$\supp a_i\subset A^{i_0}B^{(i)}$ with some fixed $i_0\in\zz$,
$$\|a_i\|_{L^q(\rn)}
\le\frac{|B^{(i)}|^{1/q}}{\|\mathbf{1}_{B^{(i)}}\|_{L^{r(\cdot)}(\rn)}}$$
and
$$\lf\|\lf\{\sum_{i\in\nn}\lf[\frac{|\lz_i|\mathbf{1}_{B^{(i)}}}
{\|\mathbf{1}_{B^{(i)}}\|_{L^{r(\cdot)}(\rn)}}\r]^
{\underline{r}}\r\}^{1/\underline{r}}\r\|_{L^{r(\cdot)}(\rn)}<\fz,$$
where $\underline{r}$ is as in \eqref{2e3}. Then
$$\lf\|\lf[\sum_{i\in\nn}\lf|\lz_ia_i\r|^{\underline{r}}\r]
^{1/\underline{r}}\r\|_{L^{r(\cdot)}(\rn)}
\le C\lf\|\lf\{\sum_{i\in\nn}\lf[\frac{|\lz_i|\mathbf{1}_{B^{(i)}}}
{\|\mathbf{1}_{B^{(i)}}\|_{L^{r(\cdot)}(\rn)}}\r]^
{\underline{r}}\r\}^{1/\underline{r}}\r\|_{L^{r(\cdot)}(\rn)},$$
where $C$ is a positive constant independent of $\lz_i$, $B^{(i)}$ and $a_i$.
\end{lemma}

The succeeding notions of both anisotropic $(p(\cdot),r,s)$-atoms and
variable anisotropic atomic Hardy spaces $\vah$ are from \cite{lwyy18}.

\begin{definition}\label{3d3}
\begin{enumerate}
\item[{\rm (i)}]
Let $p(\cdot)$, $r$ and $s$ be as in Definition \ref{3d1}.
An \emph{anisotropic $(p(\cdot),r,s)$-atom} (shortly, a $(p(\cdot),r,s)$-\emph{atom})
is a measurable function $a$ on $\rn$ satisfying
\begin{enumerate}
\item[{\rm (i)$_1$}] $\supp a \subset B$, where
$B\in\B$ with $\B$ as in \eqref{2e1};

\item[{\rm (i)$_2$}] $\|a\|_{L^r(\rn)}\le \frac{|B|^{1/r}}{\|\mathbf{1}_B\|_{\lv}}$;

\item[{\rm (i)$_3$}] $\int_{\rn}a(x)x^\gz\,dx=0$ for any $\gz\in\zz_+^n$
with $|\gz|\le s$.
\end{enumerate}
\item[{\rm (ii)}]
Let $p(\cdot)\in C^{\log}(\rn)$, $r\in(1,\fz]$ and $s$ be as in \eqref{3e2}.
The \emph{variable anisotropic atomic Hardy space} $\vah$ is defined to be the
set of all $f\in\cs'(\rn)$ satisfying that there exist a sequence
$\{\lz_i\}_{i\in\nn}\subset\mathbb{C}$ and a sequence of $(p(\cdot),r,s)$-atoms,
$\{a_i\}_{i\in\nn}$, supported, respectively, in
$\{B^{(i)}\}_{i\in\nn}\subset\B$ such that
\begin{align*}
f=\sum_{i\in\nn}\lz_ia_i
\quad\mathrm{in}\quad\cs'(\rn).
\end{align*}
Moreover, for any $f\in\vah$, let
\begin{align*}
\|f\|_{\vah}:=
{\inf}\lf\|\lf\{\sum_{i\in\nn}
\lf[\frac{|\lz_i|\mathbf{1}_{B^{(i)}}}{\|\mathbf{1}_{B^{(i)}}\|_{\lv}}\r]^
{\underline{p}}\r\}^{1/\underline{p}}\r\|_{\lv},
\end{align*}
where the infimum is taken over all decompositions of $f$ as above.
\end{enumerate}
\end{definition}

We also need the following atomic characterizations of $\vh$ established
in \cite[Theorem 4.8]{lwyy18}.

\begin{lemma}\label{3l3}
Let $p(\cdot)\in C^{\log}(\rn)$, $r\in(\max\{p_+,1\},\fz]$ with $p_+$
as in \eqref{2e3}, $s$ be as in \eqref{3e2} and
$N\in\mathbb{N}\cap[\lfloor(\frac1{\underline{p}}-1)\frac{\ln b}{\ln
\lambda_-}\rfloor+2,\fz)$ with $\underline{p}$ as in \eqref{2e3}.
Then $\vh=\vah$ with equivalent quasi-norms.
\end{lemma}

The following two lemmas are, respectively, from \cite[Remark 2.1(i)]{yyyz16} and
\cite[p.\,8, (2.11), p.\,5, (2.1) and (2.2) and p.\,17, Proposition 3.10]{mb03}.

\begin{lemma}\label{3l5}
Let $p(\cdot)\in\cp(\rn)$. Then, for any $s\in (0,\fz)$ and $f\in\lv$,
$$\lf\||f|^s\r\|_{\lv}=\|f\|_{L^{sp(\cdot)}(\rn)}^s.$$
In addition, for any $\lz\in{\mathbb C}$ and $f,\ g\in\lv$,
$\|\lz f\|_{\lv}=|\lz|\|f\|_{\lv}$ and
$$\|f+g\|_{\lv}^{\underline{p}}\le \|f\|_{\lv}^{\underline{p}}
+\|g\|_{\lv}^{\underline{p}},$$
where $\underline{p}$ is as in \eqref{2e3}.
\end{lemma}

\begin{lemma}\label{3l4}
Let $A$ be some fixed dilation. Then
\begin{enumerate}
\item[{\rm (i)}] for any $i\in\zz$, we have
$$
B_i+B_i\subset B_{i+\tau}\quad and\quad
B_i+(B_{i+\tau})^\com\subset (B_i)^\com,
$$
where $\tau$ is as in \eqref{2e2};

\item[{\rm (ii)}] there exists a positive constant $C$ such that,
for any $x\in\rn$, when $k\in\zz_+$,
$$
C^{-1}\lf(\lambda_-\r)^k|x|\le\lf|A^kx\r|\le C\lf(\lambda_+\r)^k|x|
$$
and, when $k\in\zz\setminus\zz_+$,
$$
C^{-1}\lf(\lambda_+\r)^k|x|\le\lf|A^kx\r|\le C\lf(\lambda_-\r)^k|x|;
$$

\item[{\rm (iii)}] for any given $N\in\nn$, there exists a positive constant
$C_{(N)}$, depending on $N$, such that, for any $f\in\cs'(\rn)$ and $x\in\rn$,
$$M_N^0(f)(x)\le M_N(f)(x)
\le C_{(N)}M_N^0(f)(x),$$
where $M_N^0$ denotes the radial grand maximal function of $f\in\cs'(\rn)$
defined by setting, for any $x\in\rn$,
$$
M_N^0(f)(x):=\sup_{\phi\in\cs_N(\rn)}\sup_{i\in\zz}|f\ast\phi_i(x)|.
$$
\end{enumerate}
\end{lemma}

Let $L_{\rm loc}^1(\rn)$ be the collection of all locally integrable functions
on $\rn$. Recall that, for any $f\in L^1_{{\rm loc}}(\rn)$, the \emph{anisotropic
Hardy--Littlewood maximal function} $M_{{\rm HL}}(f)$ is defined by setting,
for any $x\in\rn$,
\begin{align}\label{3e1}
M_{{\rm HL}}(f)(x):=\sup_{k\in\zz}\sup_{y\in x+B_k}\frac1{|B_k|}
\int_{y+B_k}|f(z)|\,dz=\sup_{x\in B\in\B}\frac1{|B|}\int_B|f(z)|\,dz,
\end{align}
where $\B$ is as in \eqref{2e1}.

The following Fefferman--Stein vector-valued inequality of the maximal
operator $\HL$ on the variable Lebesgue space $\lv$ is just \cite[Lemma 4.3]{lyy17hl}.

\begin{lemma}\label{3l6}
Let $\nu\in(1,\fz]$.
Assume that $p(\cdot)\in C^{\log}(\rn)$ satisfies $1<p_-\le p_+<\fz$.
Then there exists a positive constant $C$ such that, for any sequence
$\{f_k\}_{k\in\nn}$ of measurable functions,
$$\lf\|\lf\{\sum_{k\in\nn}
\lf[\HL(f_k)\r]^\nu\r\}^{1/\nu}\r\|_{\lv}
\le C\lf\|\lf(\sum_{k\in\nn}|f_k|^\nu\r)^{1/\nu}\r\|_{\lv}$$
with the usual modification made when $\nu=\fz$,
where $\HL$ denotes the Hardy--Littlewood maximal operator as in \eqref{3e1}.
\end{lemma}

Now we state the main result of this section as follows.

\begin{theorem}\label{3t1}
Let $p(\cdot)$, $r$, $s$ and $N$ be as in Lemma \ref{3l3} and
$\var\in((s+1)\log_b{(\lz_+/\lz_-)},\fz)$.
Then $\vh=\vmh$ with equivalent quasi-norms.
\end{theorem}

\begin{proof}
Let $p(\cdot)$, $r$ and $s$ be as in Lemma \ref{3l3}.
Then, by Lemma \ref{3l3}, we know that $\vh=\vah$ with equivalent quasi-norms.
Moreover, by the definitions of both $\vah$ and $\vmh$, we find that
$\vah\subset\vmh$ and this inclusion is continuous. Thus, $\vh\subset\vmh$
with continuous inclusion.

Conversely, for any $f\in\vmh$, without loss of generality, we can assume that $f$
is not the zero element of $\vmh$. Then, from Definition \ref{3d2}, it follows that
there exist a sequence $\{\lz_i\}_{i\in\nn}\subset\mathbb{C}$ and a sequence of
$(p(\cdot),r,s,\var)$-molecules, $\{m_i\}_{i\in\nn}$, associated, respectively, to
$\{B^{(i)}\}_{i\in\nn}\subset\B$ such that
\begin{align}\label{3e3}
f=\sum_{i\in\nn}\lz_im_i\quad {\rm in}\quad \cs'(\rn),
\end{align}
and
\begin{align*}
\|f\|_{\vmh}\sim\lf\|\lf\{\sum_{i\in\nn}
\lf[\frac{|\lz_i|\mathbf{1}_{B^{(i)}}}{\|\mathbf{1}_{B^{(i)}}\|_{\lv}}\r]^
{\underline{p}}\r\}^{1/\underline{p}}\r\|_{\lv},
\end{align*}
where $\underline{p}$ is as in \eqref{2e3}. Obviously, there exist two sequences
$\{x_i\}_{i\in\nn}\subset\rn$ and $\{\ell_i\}_{i\in\nn}\subset\zz$ such that,
for any $i\in\nn$, $x_i+B_{\ell_i}=B^{(i)}$. By \eqref{3e3}, it is easy to see that,
for any $N\in\nn\cap[\lfloor(\frac1{\underline{p}}-1)\frac{\ln b}{\ln
\lambda_-}\rfloor+2,\fz)$ and $x\in\rn$,
\begin{align}\label{3e4}
M_N(f)(x)\le\sum_{i\in\nn}|\lz_i|M_N(m_i)(x)\mathbf{1}_{x_i+A^\tau B_{\ell_i}}(x)
+\sum_{i\in\nn}|\lz_i|M_N(m_i)(x)\mathbf{1}_{(x_i+A^\tau B_{\ell_i})^\com}(x)
=:{\rm I}_1+{\rm I}_2.
\end{align}

To deal with ${\rm I}_1$, for any $r\in(\max\{p_+,1\},\fz]$, $\var$ as in Theorem \ref{3t1}
and $i\in\nn$, from the boundedness of $M_N$ on $L^r(\rn)$ (see \cite[Remark 2.10]{lyy16})
and Definition \ref{3d1}(i), we deduce that
\begin{align*}
\|M_N(m_i)\|_{L^r(\rn)}
&\ls\|m_i\|_{L^r(\rn)}\ls\sum_{j\in\zz_+}\|m_i\|_{L^r(U_j(B^{(i)}))}\\
&\ls\sum_{j\in\zz_+}b^{-j\var}\frac{|B^{(i)}|^{1/r}}{\|\mathbf{1}_{B^{(i)}}\|_{\lv}}
\sim\frac{|B^{(i)}|^{1/r}}{\|\mathbf{1}_{B^{(i)}}\|_{\lv}},
\end{align*}
where $U_0(B^{(i)}):=B^{(i)}$ and, for any $j\in\nn$,
$U_j(B^{(i)})=U_j(x_i+B_{\ell_i}):=x_i+(A^j B_{\ell_i})\setminus(A^{j-1}B_{\ell_i})$.
This, combined with Lemma \ref{3l2}, implies that
\begin{align}\label{3e5}
\|{\rm I}_1\|_{\lv}
&\ls\lf\|\lf\{\sum_{i\in\nn}\lf[|\lz_i|M_N(m_i)\mathbf{1}_{A^\tau B^{(i)}}\r]
^{\underline{p}}\r\}^{1/\underline{p}}\r\|_{\lv}\\
&\ls\lf\|\lf\{\sum_{i\in\nn}
\lf[\frac{|\lz_i|\mathbf{1}_{B^{(i)}}}{\|\mathbf{1}_{B^{(i)}}\|_{\lv}}\r]^
{\underline{p}}\r\}^{1/\underline{p}}\r\|_{\lv}\sim\|f\|_{\vmh}.\noz
\end{align}

For the term ${\rm I}_2$, suppose that $P$ is a polynomial of degree not greater than
$s$. Then, by Definition \ref{3d1} and the H\"{o}lder inequality, we find that, for any
$N\in\nn$, $\phi\in\cs_N(\rn)$, $t\in\zz$ and $x\in (x_i+B_{\ell_i+\tau})^\com$ with $i\in\nn$,
\begin{align*}
\lf|(m_i\ast\phi_t)(x)\r|
&=b^{-t}\lf|\int_\rn m_i(y)\phi\lf(A^{-t}(x-y)\r)\,dy\r|\\
&\le b^{-t}\sum_{j\in\zz_+}\lf|\int_{U_j(x_i+B_{\ell_i})}m_i(y)
\lf[\phi\lf(A^{-t}(x-y)\r)-P\lf(A^{-t}(x-y)\r)\r]\,dy\r|\noz\\
&\le b^{-t}\sum_{j\in\zz_+}\sup_{y\in A^{-t}(x-x_i)+A^jB_{\ell_i-t}}
|\phi(y)-P(y)|\int_{U_j(x_i+B_{\ell_i})}|m_i(y)|\,dy\noz\\
&\ls b^{\ell_i/r'-t}\sum_{j\in\zz_+}b^{j/r'}\sup_{y\in A^{-t}(x-x_i)+A^jB_{\ell_i-t}}
|\phi(y)-P(y)|\|m_i\|_{L^r(U_j(x_i+B_{\ell_i}))}\noz\\
&\ls b^{\ell_i-t}\lf\|\mathbf{1}_{x_i+B_{\ell_i}}\r\|_{\lv}^{-1}\sum_{j\in\zz_+}
b^{(1/r'-\var)j}\sup_{y\in A^{-t}(x-x_i)+A^jB_{\ell_i-t}}|\phi(y)-P(y)|.\noz
\end{align*}
Assume that $x\in [x_i+(B_{\ell_i+\tau+k+1}\setminus B_{\ell_i+\tau+k})]$
for some $k\in\zz_+$. Without loss of generality, we can assume that
$s=\lfloor(1/\underline{p}-1)\ln b/\ln\lz_-\rfloor$ and $N=s+2$, which implies that
$b\lz_-^{s+1}\le b^N$. Then, by Lemma \ref{3l4}, an estimation similar to that
used in the proof of \cite[(3.9)]{lhy18} and the fact that
$\var\in((s+1)\log_b(\lz_+/\lz_-),\fz)$, we conclude that, for any $i\in\nn$,
\begin{align}\label{3e7}
M_N(m_i)(x)
&\ls\lf\|\mathbf{1}_{B^{(i)}}\r\|_{\lv}^{-1}\sum_{j\in\zz_+}
\lf\{b^{-j\var}+b^{-j[\var-(s+1)\log_b(\lz_+/\lz_-)]}\r\}
\max\lf\{b^{-Nk},\,\lf(b\lz_-^{s+1}\r)^{-k}\r\}\\
&\ls\lf\|\mathbf{1}_{B^{(i)}}\r\|_{\lv}^{-1}\lf(b\lz_-^{s+1}\r)^{-k}
\sim\lf\|\mathbf{1}_{B^{(i)}}\r\|_{\lv}^{-1}
b^{-k}b^{-(s+1)k\frac{\ln\lambda_-}{\ln b}}\noz\\
&\ls\lf\|\mathbf{1}_{B^{(i)}}\r\|_{\lv}^{-1}
b^{\ell_i[(s+1)\frac{\ln\lambda_-}{\ln b}+1]}
b^{-(\ell_i+\tau+k)[(s+1)\frac{\ln\lambda_-}{\ln b}+1]}\noz\\
&\ls\lf\|\mathbf{1}_{B^{(i)}}\r\|_{\lv}^{-1}
\frac{|B^{(i)}|^\beta}{[\rho(x-x_i)]^\beta}
\ls\lf\|\mathbf{1}_{B^{(i)}}\r\|_{\lv}^{-1}
\lf[\HL\lf(\mathbf{1}_{B^{(i)}}\r)(x)\r]^\bz,\noz
\end{align}
where
\begin{align*}
\beta:=\lf(\frac{\ln b}{\ln \lambda_-}+s+1\r)
\frac{\ln \lambda_-}{\ln b}>\frac1{\underline{p}}.
\end{align*}
From this and Lemmas \ref{3l5} and \ref{3l6}, it follows that
\begin{align*}
\|{\rm I}_2\|_{\lv}
&\ls\lf\|\sum_{i\in\nn}\frac{|\lz_i|}{\|\mathbf{1}_{B^{(i)}}\|_{\lv}}
\lf[\HL\lf(\mathbf{1}_{B^{(i)}}\r)\r]^\bz\r\|_{\lv}\\
&\sim\lf\|\lf\{\sum_{i\in\nn}\frac{|\lz_i|}{\|\mathbf{1}_{B^{(i)}}\|_{\lv}}
\lf[\HL(\mathbf{1}_{B^{(i)}})\r]^\bz\r\}^{1/\bz}\r\|_{L^{\bz p(\cdot)}}^{\bz}\\
&\ls\lf\|\sum_{i\in\nn}
\frac{|\lz_i|\mathbf{1}_{B^{(i)}}}{\|\mathbf{1}_{B^{(i)}}\|_{\lv}}\r\|_{\lv}
\ls\lf\|\lf\{\sum_{i\in\nn}
\lf[\frac{|\lz_i|\mathbf{1}_{B^{(i)}}}{\|\mathbf{1}_{B^{(i)}}\|_{\lv}}\r]^
{\underline{p}}\r\}^{1/\underline{p}}\r\|_{\lv}\\
&\sim\|f\|_{\vmh}.
\end{align*}
This, together with \eqref{3e4}, \eqref{3e5} and Lemma \ref{3l5} again, implies that
$$\|f\|_{\vh}=\|M_N(f)\|_{\lv}\ls\|f\|_{\vmh},$$
which completes the proof of Theorem \ref{3t1}.
\end{proof}

\begin{remark}\label{3r1}
\begin{enumerate}
\item[(i)] When $A:=d\,{\rm I}_{n\times n}$ for some $d\in\rr$ with $|d|\in(1,\fz)$,
here and thereafter, ${\rm I}_{n\times n}$ denotes the $n\times n$ \emph{unit matrix},
$\vh$ and $\vmh$ become, respectively, the classical isotropic variable Hardy
space (see \cite{cw14,ns12}) and variable molecular Hardy space. In this case,
Theorem \ref{3t1} includes \cite[Theorem 5.2]{ns12} as a special case.
\item[(ii)] Recall that, in \cite[Theorem 3.12]{lhy18}, the authors established
the molecular characterizations of the anisotropic Musielak--Orlicz Hardy space
$H_A^\varphi(\rn)$ with $\varphi:\ \rn\times[0,\fz)\to[0,\fz)$ being an anisotropic
growth function (see \cite[Definition 3]{lyy14}). By \cite[Remark 2.5(iii)]{lwyy18},
we know that the anisotropic Musielak--Orlicz Hardy space $\hvz$ and the variable
anisotropic Hardy space $\vh$ in this article can not cover each other, and hence
neither do \cite[Theorem 3.12]{lhy18} and Theorem \ref{3t1}.
\item[(iii)] Very recently, in \cite[Theorem 2.8]{lql}, Liu et al. obtained
the molecular characterizations of variable anisotropic Hardy--Lorentz spaces
$\vhlpq$ with $p(\cdot)\in C^{\log}(\rn)$ and $q\in(0,\fz]$. It is easy to see
that the space $\vh$, in this article, is not covered by the space $\vhlpq$
since the exponent $q\in(0,\fz]$ in $\vhlpq$ is only a constant. Thus,
Theorem \ref{3t1} is neither covered by \cite[Theorem 2.8]{lql}.
\item[(iv)] We should also point out that \cite[Theorem 3.12]{lhy18} and
\cite[Theorem 2.8]{lql}, respectively, require the decay index $\var$ to belong to
$$\lf(\max\lf\{1,\,(s+1)\log_b{(\lz_+/\lz_-)}\r\},\fz\r)\quad{\rm and}\quad
\lf(\max\lf\{1,\,(s+1)\log_b{(\lz_+)}\r\},\fz\r),$$
either of which is just a proper subset of
$$((s+1)\log_b{(\lz_+/\lz_-)},\fz)$$
from Theorem \ref{3t1}. In this sense, we improve the range of $\var$ in the
molecular characterizations. In particular, when $A$ is as in (i) of this remark
and $p(\cdot)\equiv p\in(0,\fz)$, the space $\vh$ becomes the classical isotropic
Hardy space $H^p(\rn)$ and $\log_b{(\lz_+/\lz_-)}=0$. In this case, Theorem \ref{3t1}
gives a molecular characterization of $H^p(\rn)$ with the known best possible decay
of molecules, namely, $\var\in(0,\fz)$.
\end{enumerate}
\end{remark}

\section{Some applications\label{s4}}

In this section, as applications, we first establish a criterion on the boundedness
of linear operators on $\vh$. Then, applying this criterion, we obtain the boundedness of
anisotropic Calder\'on--Zygmund operators on $\vh$. In addition, the boundedness
of these operators from $\vh$ to the variable Lebesgue space $\lv$ is also presented.

First, we recall the notion of
variable anisotropic finite atomic Hardy spaces $\vfah$ from \cite[Definition 5.1]{lwyy18}.

\begin{definition}\label{4d1}
Let $p(\cdot)\in C^{\log}(\rn)$, $r\in(1,\fz]$, $s$ be as in \eqref{3e2} and $A$ a dilation.
The \emph{variable anisotropic finite atomic Hardy space} $\vfah$ is defined to be the set
of all $f\in\cs'(\rn)$ satisfying that there exist an $I\in\nn$, a finite sequence
$\{\lz_i\}_{i\in[1,I]\cap\nn}\subset\cc$ and a finite sequence of $(p(\cdot),r,s)$-atoms,
$\{a_i\}_{i\in[1,I]\cap\nn}$, supported, respectively, in
$\{B^{(i)}\}_{i\in[1,I]\cap\nn}\subset\mathfrak{B}$ such that
\begin{align*}
f=\sum_{i=1}^I\lambda_ia_i
\quad\mathrm{in}\quad\cs'(\rn).
\end{align*}
Moreover, for any $f\in\vfah$, let
\begin{align*}
\|f\|_{\vfah}:=
{\inf}\lf\|\lf\{\sum_{i=1}^{I}
\lf[\frac{|\lz_i|\mathbf{1}_{B^{(i)}}}{\|\mathbf{1}_{B^{(i)}}\|_{\lv}}\r]^
{\underline{p}}\r\}^{1/\underline{p}}\r\|_{\lv}
\end{align*}
with $\underline{p}$ as in \eqref{2e3}, where the
infimum is taken over all decompositions of $f$ as above.
\end{definition}

Denote by $C(\rn)$ the \emph{set of all continuous functions} and by $C_{\rm c}^\fz(\rn)$
the \emph{set of all infinite differentiable functions with compact support}. The
succeeding three lemmas are just, respectively, \cite[Theorem 5.4 and Lemma 7.3]{lwyy18}
and \cite[Lemma 5.4]{lyy17hl}.

\begin{lemma}\label{4l1}
Let $p(\cdot)\in C^{\log}(\rn)$, $r\in(\max\{p_+,1\},\fz]$ and $s$ be as in \eqref{3e2},
where $p_+$ is as in \eqref{2e3}.
\begin{enumerate}
\item[{\rm (i)}]
If $r\in(\max\{p_+,1\},\fz)$, then $\|\cdot\|_{\vfah}$
and $\|\cdot\|_{\vh}$ are two equivalent quasi-norms on $\vfah$;
\item[{\rm (ii)}]
$\|\cdot\|_{\vfahfz}$
and $\|\cdot\|_{\vh}$ are two equivalent quasi-norms on $\vfahfz\cap C(\rn)$.
\end{enumerate}
\end{lemma}

\begin{lemma}\label{4l2}
Let $p(\cdot)\in C^{\log}(\rn)$. Then $\vh\cap C_{\rm c}^\fz(\rn)$ is dense in $\vh$.
\end{lemma}

\begin{lemma}\label{4l3}
Let $p(\cdot)\in C^{\log}(\rn)$ and $p_-\in(1,\fz)$.
Then there exists a positive constant $C$ such that,
for arbitrary two subsets $E_1$, $E_2$ of $\rn$ with $E_1\subset E_2$,
\begin{align*}
C^{-1}\lf(\frac{|E_1|}{|E_2|}\r)^{\frac 1{p_-}}
\le\frac{\|\mathbf{1}_{E_1}\|_{L^{p(\cdot)}(\rn)}}
{\|\mathbf{1}_{E_2}\|_{L^{p(\cdot)}(\rn)}}
\le C\lf(\frac{|E_1|}{|E_2|}\r)^{\frac 1{p_+}}.
\end{align*}
\end{lemma}

Applying the above three lemmas and Theorem \ref{3t1}, we establish a criterion on
the boundedness of linear operators on $\vh$ as follows.

\begin{theorem}\label{4t1}
Assume that $T$ is a linear operator defined on the set of all measurable functions.
Let $p(\cdot)\in C^{\log}(\rn)$, $r\in(\max\{p_+,1\},\fz]$ with $p_+$
as in \eqref{2e3} and $\widetilde{s}$
be as in \eqref{3e2} with $s$ replaced by $\widetilde{s}$.
If there exist some $k_0\in\zz$ and a positive constant $C$ such that,
for any $(p(\cdot),r,\widetilde{s})$-atom $\wz a$ supported in
some dilated ball $x_0+B_{i_0}\in\B$ with $x_0\in\rn$, $i_0\in\zz$ and $\B$ as in
\eqref{2e1}, $\frac1C T(\wz a)$ is a $(p(\cdot),r,s,\var)$-molecule associated to
$x_0+B_{i_0+k_0}$, where $s$ and $\var$ are as in Theorem \ref{3t1}, then $T$ has a
unique bounded linear extension on $\vh$.
\end{theorem}

\begin{proof}
Let $p(\cdot)\in C^{\log}(\rn)$, $r\in(\max\{p_+,1\},\fz]$ and
$\widetilde{s}\in[\lfloor(1/{p_-}-1)
{\ln b}/{\ln\lz_-}\rfloor,\fz)\cap\zz_+$ with $p_-$ as in \eqref{2e3}. We next prove
Theorem \ref{4t1} by considering two cases.

\emph{Case 1).} $r\in(\max\{p_+,1\},\fz)$. In this case, for any
$f\in H_{A,\,{\rm fin}}^{p(\cdot),r,\widetilde{s}}(\rn)$, by Definition \ref{4d1},
we know that there exist some $I\in\nn$, three finite sequences
$\{\lz_i\}_{i\in[1,I]\cap\nn}\subset\mathbb{C}$, $\{x_i\}_{i\in[1,I]\cap\nn}\subset\rn$
and $\{\ell_i\}_{i\in[1,I]\cap\nn}\subset\zz$, and a finite sequence of
$(p(\cdot),r,\wz{s})$-atoms, $\{a_i\}_{i\in[1,I]\cap\nn}$, supported,
respectively, in $\{x_i+B_{\ell_i}\}_{i\in[1,I]\cap\nn}\subset\B$
such that $f=\sum_{i=1}^I\lz_ia_i$ in $\cs'(\rn)$ and
\begin{align}\label{4e1}
\|f\|_{H_{A,\,{\rm fin}}^{p(\cdot),r,\widetilde{s}}(\rn)}\sim\lf\|\lf\{\sum_{i=1}^{I}
\lf[\frac{|\lz_i|\mathbf{1}_{x_i+B_{\ell_i}}}{\|\mathbf{1}_{x_i+B_{\ell_i}}\|_{\lv}}\r]^
{\underline{p}}\r\}^{1/\underline{p}}\r\|_{\lv}.
\end{align}
From this and the linearity of $T$, it is easy to see that
$T(f)=\sum_{i=1}^I\lz_iT(a_i)$ in $\cs'(\rn)$, where, for any $i\in[1,I]\cap\nn$,
$\frac1C T(a_i)$ with $C$ being a positive constant independent of $i$ is
a $(p(\cdot),r,s,\var)$-molecule associated to $x_i+B_{\ell_i+k_0}$ with $s$, $\var$
and $k_0$ as in Theorem \ref{4t1}. By this, Theorem \ref{3t1}, Definition \ref{3d2},
Lemmas \ref{4l3}, \ref{3l6} and \ref{3l5}, \eqref{4e1} and Lemma \ref{4l1},
we further conclude that, for any
$f\in H_{A,\,{\rm fin}}^{p(\cdot),r,\widetilde{s}}(\rn)$,
\begin{align}\label{4e2}
\|T(f)\|_{\vh}&\sim\|T(f)\|_{\vmh}
\ls\lf\|\lf\{\sum_{i=1}^{I}\lf[\frac{|\lz_i|
\mathbf{1}_{x_i+B_{\ell_i+k_0}}}{\|\mathbf{1}_{x_i+B_{\ell_i+k_0}}\|_
{\lv}}\r]^{\underline{p}}\r\}^{1/\underline{p}}\r\|_{\lv}\\
&\ls b^{k_0/\omega}\lf\|\lf[\sum_{i=1}^{I}\lf\{\frac{|\lz_i|
[\HL(\mathbf{1}_{x_i+B_{\ell_i}})]^{1/\omega}}{\|\mathbf{1}_{x_i+B_{\ell_i}}\|_
{\lv}}\r\}^{\underline{p}}\r]^{1/\underline{p}}\r\|_{\lv}\noz\\
&\sim b^{k_0/\omega}\lf\|\lf\{\sum_{i=1}^{I}\lf[\frac{|\lz_i|^\omega
\HL(\mathbf{1}_{x_i+B_{\ell_i}})}{\|\mathbf{1}_{x_i+B_{\ell_i}}\|_{\lv}^\omega}\r]^
{\underline{p}/\omega}\r\}^{\omega/\underline{p}}\r\|_
{L^{p(\cdot)/\omega}(\rn)}^{1/\omega}\noz\\
&\ls\lf\|\lf\{\sum_{i=1}^{I}\lf[\frac{|\lz_i|
\mathbf{1}_{x_i+B_{\ell_i}}}{\|\mathbf{1}_{x_i+B_{\ell_i}}\|_{\lv}}\r]^
{\underline{p}}\r\}^{1/\underline{p}}\r\|_{\lv}
\sim\|f\|_{H_{A,\,{\rm fin}}^{p(\cdot),r,\widetilde{s}}(\rn)}\sim\|f\|_{\vh},\noz
\end{align}
where $\omega\in(0,\underline{p})$ is a constant.

Now let $f\in\vh$. Then, by the obvious density of
$H_{A,\,{\rm fin}}^{p(\cdot),r,\widetilde{s}}(\rn)$ in $\vh$ with respect to the
quasi-norm $\|\cdot\|_{\vh}$, we find that there exists a Cauchy sequence
$\{f_j\}_{j\in\nn}\subset H_{A,\,{\rm fin}}^{p(\cdot),r,\widetilde{s}}(\rn)$
such that
$$\lim_{j\to\fz}\lf\|f_j-f\r\|_{\vh}=0.$$
This, combined with the linearity of $T$ and \eqref{4e2}, implies that,
as $j$, $m\to\fz$,
\begin{align*}
\lf\|T(f_j)-T(f_m)\r\|_{\vh}=\lf\|T(f_j-f_m)\r\|_{\vh}
\ls\lf\|f_j-f_m\r\|_{\vh}\to0.
\end{align*}
Thus, $\{T(f_j)\}_{j\in\nn}$ is also a Cauchy sequence in $\vh$. From this and the
completeness of $\vh$, it follows that there exists some $g\in\vh$ such that
$g=\lim_{j\to\fz}T(f_j)$ in $\vh$. Then let $T(f):=g$. By this and \eqref{4e2},
we know that $T(f)$ is well defined and, moreover, for any $f\in\vh$,
\begin{align}\label{4e3}
\|T(f)\|_{\vh}
&\ls\limsup_{j\to\fz}\lf[\lf\|T(f)-T(f_j)\r\|_{\vh}
+\lf\|T(f_j)\r\|_{\vh}\r]\\
&\sim\limsup_{j\to\fz}\lf\|T(f_j)\r\|_{\vh}
\ls\lim_{j\to\fz}\lf\|f_j\r\|_{\vh}\sim\|f\|_{\vh},\noz
\end{align}
which completes the proof of Theorem \ref{4t1} in Case 1).

\emph{Case 2).} $r=\fz$. In this case, by Lemmas \ref{4l2} and \ref{3l3}, it is easy
to see that $H_{A,\,{\rm fin}}^{p(\cdot),\fz,\widetilde{s}}(\rn)\cap C(\rn)$ is
dense in $\vh$. By this, repeating the proof of Case 1) with some slight modifications,
we conclude that Theorem \ref{4t1} also holds true when $r=\fz$. This finishes the proof
of Theorem \ref{4t1}.
\end{proof}

Next we consider the boundedness of anisotropic Calder\'{o}n--Zygmund
operators from $\vh$ to itself [or to $\lv$]. To this end, we first recall the notion
of anisotropic Calder\'{o}n--Zygmund operators from \cite[p.\,60, Definition 9.1]{mb03}
as follows.

\begin{definition}\label{2d8}
A locally integrable function $\mathcal{K}$ on
$\Omega:=\{(x,y)\in\rn\times\rn:\ x\neq y\}$ is called an \emph{anisotropic
Calder\'{o}n--Zygmund standard kernel} if there exist two positive constants $C$
and $\dz$ such that, for any $x,\,y,\,\wz{x},\,\wz{y}\in\Omega$,
$$|\mathcal{K}(x,y)|\le \frac C{\rho(x-y)}\quad{\rm when}\quad x\neq y,$$
\begin{align*}
|\mathcal{K}(x,y)-\mathcal{K}(x,\wz{y})|\le C\frac{[\rho(y-\wz{y})]^\dz}
{[\rho(x-y)]^{1+\dz}}
\quad{\rm when}\quad \rho(x-y)\ge b^{2\tau}\rho(y-\wz{y}),
\end{align*}
and
\begin{align*}
|\mathcal{K}(x,y)-\mathcal{K}(\wz{x},y)|\le C\frac{[\rho(x-\wz{x})]^\dz}
{[\rho(x-y)]^{1+\dz}}
\quad{\rm when}\quad \rho(x-y)\ge b^{2\tau}\rho(x-\wz{x}),
\end{align*}
with $\tau$ as in \eqref{2e2}. Moreover, a linear operator $T$ is called an
\emph{anisotropic Calder\'on--Zygmund operator} if it is bounded on $L^2(\rn)$
and there exists an anisotropic Calder\'{o}n--Zygmund standard kernel $\mathcal{K}$
such that, for any $f\in L^2(\rn)$ with compact support and $x\notin\supp f$,
$$T(f)(x)=\int_{\supp f}\mathcal{K}(x,y)f(y)\,dy.$$
\end{definition}

In what follows, for any $\nu\in\nn$, denote by $C^\nu(\rn)$ the \emph{set of all
functions on $\rn$ whose derivatives with order not greater than $\nu$ are continuous}.
Since we are interested in the boundedness of anisotropic Calder\'{o}n--Zygmund
operators on $\vh$ with $p(\cdot)\in C^{\log}(\rn)$, we need to increase the smooth
hypothesis on the corresponding kernel $\mathcal{K}$ as follows, which originates
from \cite[p.\,61, Definition 9.2]{mb03}.

\begin{definition}\label{4d2}
Let $\nu\in\nn$. An anisotropic Calder\'on--Zygmund operator $T$ is called an
\emph{anisotropic Calder\'on--Zygmund operator of order} $\nu$ if its kernel
$\mathcal{K}$ is a $C^\nu(\rn)$ function with respect to the second variable $y$ and
there exists a positive constant $C$ such that, for any $\az\in\zz_+^n$ with
$1\le|\az|\le \nu$, $m\in\zz$ and $x,\,y\in\Omega$ with $\rho(x-y)=b^m$,
\begin{align}\label{4e4}
\lf|\partial^\az_y\widetilde{\mathcal{K}}\lf(x,\,A^{-m}y\r)\r|
\le C[\rho(x-y)]^{-1}=Cb^{-m},
\end{align}
where, for any $x,\,y\in\rn$ satisfying $x\neq A^m y$,
$\wz{\mathcal{K}}(x,y):=\mathcal{K}(x,\,A^m y)$.
\end{definition}

\begin{remark}\label{4r1}
By \cite[Remark 4.3]{lql}, we know that, for any $\nu\in\nn$,
the classical isotropic Calder\'on--Zygmund operator of order $\nu$
(see \cite[p.\,289]{s93}) is an operator as in Definition \ref{4d2} in the
case when $A:=d\,{\rm I}_{n\times n}$ for some $d\in\rr$ with $|d|\in(1,\fz)$.
For more details related to the kernels satisfying
\eqref{4e4}, we refer the reader to \cite[p.\,61, Example]{mb03}.
\end{remark}

Motivated by \cite[p.\,64, Definition 9.4]{mb03}, we introduce the following
vanishing moment condition.

\begin{definition}\label{4d3}
Let $p(\cdot)\in C^{\log}(\rn)$, $\nu\in\nn$, $s_0:=\lfloor(1/{p_-}-1)
{\ln b}/{\ln\lz_-}\rfloor$ with $p_-$ as in \eqref{2e3} and
\begin{align*}
\frac{1}{p_-}-1<\frac{(\ln\lz_-)^2}{\ln b\ln\lz_+}\nu.
\end{align*}
An anisotropic Calder\'on--Zygmund operator $T$ of order $\nu$ is said
to satisfy $T^*(x^\az)=0$ for any $\az\in\zz_+^n$ with $|\az|\le s_0$
if, for any $f\in L^2(\rn)$ with compact support and satisfying that, for any
$\gamma\in\zz_+^n$ with $|\gamma|\le \nu$, $\int_{\rn}f(x)x^{\gamma}\,dx=0$,
it holds true that, for any $\az\in\zz_+^n$
with $|\az|\le s_0$,  $\int_{\rn}T(f)(x)x^{\az}\,dx=0$.
\end{definition}

The following useful conclusion is just \cite[Lemma 4.10]{lql}.

\begin{lemma}\label{4l4}
Let $p(\cdot)$, $\nu$, $s_0$ be as in Definition \ref{4d3}.
Assume that $r\in(1,\fz]$ and $T$ is an anisotropic Calder\'on--Zygmund operator
of order $\nu$ satisfying $T^*(x^\az)=0$ for any $\az\in\zz_+^n$ with $|\az|\le s_0$.
Then there exists a positive constant $C$ such that, for any $(p(\cdot),r,\nu-1)$-atom
$\wz a$ supported in some dilated ball $x_0+B_{i_0}\in\B$ with $x_0\in\rn$, $i_0\in\zz$
and $\B$ as in \eqref{2e1}, $\frac1C T(\widetilde{a})$ is a
$(p(\cdot),r,s_0,\var)$-molecule associated to $x_0+B_{i_0+\tau+1}$, where
\begin{align*}
\var:=\nu\log_b(\lz_-)+1/r'.
\end{align*}
and $\tau$ is as in \eqref{2e2}.
\end{lemma}

Then we have the following boundedness of anisotropic Calder\'on--Zygmund
operators from $\vh$ to itself (see Theorem \ref{4t2} below) or to $\lv$
(see Theorem \ref{4t3} below).

\begin{theorem}\label{4t2}
Let $p(\cdot)$, $\nu$, $s_0$ be as in Definition \ref{4d3}.
Assume that $T$ is an anisotropic Calder\'on--Zygmund operator of order $\nu$
and satisfies $T^*(x^\az)=0$ for any $\az\in\zz_+^n$ with $|\az|\le s_0$.
Then there exists a positive constant $C$ such that, for any $f\in \vh$,
$$\|T(f)\|_{\vh}\le C\|f\|_{\vh}.$$
\end{theorem}

\begin{proof}
Indeed, Theorem \ref{4t2} is an immediate corollary of Theorem \ref{4t1} and
Lemma \ref{4l4}. This finishes the proof of Theorem \ref{4t2}.
\end{proof}

\begin{theorem}\label{4t3}
Let $p(\cdot)\in C^{\log}(\rn)$.
Assume that $T$ is an anisotropic Calder\'on--Zygmund operator of
order $\nu$ with $\nu\in[s_0+1,\fz)$, where $s_0$ is as in Definition \ref{4d3}.
Then there exists a positive constant $C$ such that, for any $f\in \vh$,
\begin{align}\label{4e6}
\|T(f)\|_{\lv}\le C\|f\|_{\vh}.
\end{align}
\end{theorem}

\begin{proof}
Let $p(\cdot)\in C^{\log}(\rn)$, $r\in(\max\{1,p_+\},\fz)$ and $s\in[\lfloor(1/{p_-}-1)
{\ln b}/{\ln\lz_-}\rfloor,\fz)\cap\zz_+$, where $p_+$ and $p_-$ are as in \eqref{2e3}.
We next prove this theorem by two steps.

\emph{Step 1).} In this step, we show that,
for any $f\in H_{A,\,{\rm fin}}^{p(\cdot),r,s}(\rn)$, \eqref{4e6} holds true.
To this end, for any
$f\in H_{A,\,{\rm fin}}^{p(\cdot),r,s}(\rn)$, from Definition \ref{4d1},
it follows that there exist some $I\in\nn$, three finite sequences
$\{\lz_i\}_{i\in[1,I]\cap\nn}\subset\cc$, $\{x_i\}_{i\in[1,I]\cap\nn}\subset\rn$
and $\{\ell_i\}_{i\in[1,I]\cap\nn}\subset\zz$, and a finite sequence of
$(p(\cdot),r,s)$-atoms, $\{a_i\}_{i\in[1,I]\cap\nn}$, supported,
respectively, in $\{x_i+B_{\ell_i}\}_{i\in[1,I]\cap\nn}\subset\B$
such that $f=\sum_{i=1}^I\lz_ia_i$ in $\cs'(\rn)$ and
\begin{align}\label{4e5}
\|f\|_{H_{A,\,{\rm fin}}^{p(\cdot),r,s}(\rn)}\sim\lf\|\lf\{\sum_{i=1}^{I}
\lf[\frac{|\lz_i|\mathbf{1}_{x_i+B_{\ell_i}}}{\|\mathbf{1}_{x_i+B_{\ell_i}}\|_{\lv}}\r]^
{\underline{p}}\r\}^{1/\underline{p}}\r\|_{\lv}.
\end{align}
By the linearity of $T$ and Lemma \ref{3l5}, it is easy to see that
\begin{align}\label{4e7}
\lf\|T(f)\r\|_{\lv}
&\ls\lf\|\sum_{i=1}^I |\lz_i|T(a_i)\mathbf{1}_{x_i+B_{\ell_i+\tau}} \r\|_{\lv}
+\lf\|\sum_{i=1}^I |\lz_i|T(a_i)\mathbf{1}_{(x_i+B_{\ell_i+\tau})^\com} \r\|_{\lv} \\
& =:{\rm I}_1+{\rm I}_2.\noz
\end{align}

For ${\rm I}_1$, choose $h\in L^{(p(\cdot)/\underline{p})'}(\rn)$ satisfying that
$\|h\|_{L^{(p(\cdot)/\underline{p})'}(\rn)}\le 1$ and
$$\lf\|\sum_{i=1}^I|\lz_i|^{\underline{p}}\lf[T(a_i)\r]^{\underline{p}}
\mathbf{1}_{x_i+B_{\ell_i+\tau}}\r\|_{L^{p(\cdot)/\underline{p}}(\rn)}
=\int_{\rn} \sum_{i=1}^I|\lz_i|^{\underline{p}}
\lf[T(a_i)(x)\r]^{\underline{p}}\mathbf{1}_{x_i+B_{\ell_i+\tau}}(x)h(x)\,dx.
$$
Then, by Lemma \ref{3l5} and the H\"{o}lder inequality, we find that,
for any $t\in(1,\fz)$ with $p_+<t\underline{p}<r$,
\begin{align*}
({\rm I}_1)^{\underline{p}}
&\ls\lf\|\sum_{i=1}^I|\lz_i|^{\underline{p}}\lf[T(a_i)\r]^{\underline{p}}
\mathbf{1}_{x_i+B_{\ell_i+\tau}}\r\|_{L^{p(\cdot)/\underline{p}}(\rn)}\\
&\sim\int_{\rn} \sum_{i=1}^I|\lz_i|^{\underline{p}}
\lf[T(a_i)(x)\r]^{\underline{p}}\mathbf{1}_{x_i+B_{\ell_i+\tau}}(x)h(x)\,dx.\\
&\ls\sum_{i=1}^I|\lz_i|^{\underline{p}}
\lf\|\lf[T(a_i)\r]^{\underline{p}}\mathbf{1}_{x_i+B_{\ell_i+\tau}}\r\|_{L^t(\rn)}
\lf\|\mathbf{1}_{x_i+B_{\ell_i+\tau}} h\r\|_{L^{t'}(\rn)}\\
&\ls\sum_{i=1}^I|\lambda_i|^{\underline{p}}\lf\|T(a_i)\r\|_{L^r(\rn)}^{\underline{p}}
\lf\|\mathbf{1}_{x_i+B_{\ell_i+\tau}}\r\|_{L^{r/(r-t\underline{p})}(\rn)}^{1/t}
\lf\|\mathbf{1}_{x_i+B_{\ell_i+\tau}} h\r\|_{L^{t'}(\rn)}.
\end{align*}
By this, the boundedness of $T$ on $L^u(\rn)$ for any $u\in(1,\fz)$
(see \cite[p.\,60]{mb03}),
Definition \ref{3d3} and the H\"{o}lder inequality again, we conclude that
\begin{align*}
({\rm I}_1)^{\underline{p}}
&\ls\sum_{i=1}^I|\lz_i|^{\underline{p}}\lf\|\mathbf{1}_{x_i+B_{\ell_i}}\r\|_
{\lv}^{-\underline{p}}\lf|B_{\ell_i}\r|^{\underline{p}/r}
\lf|B_{\ell_i+\tau}\r|^{(r-t\underline{p})/{rt}}
\lf\|\mathbf{1}_{x_i+B_{\ell_i+\tau}}h\r\|_{L^{t'}(\rn)}\\
&\sim\sum_{i=1}^I|\lz_i|^{\underline{p}}
\lf\|\mathbf{1}_{x_i+B_{\ell_i}}\r\|_{\lv}^{-\underline{p}}
\lf|B_{\ell_i+\tau}\r|\lf[\frac{1}{|B_{\ell_i+\tau}|}
\int_{x_i+B_{\ell_i+\tau}}[h(x)]^{t'}\,dx\r]^{1/t'}\\
&\ls\sum_{i=1}^I|\lz_i|^{\underline{p}}
\lf\|\mathbf{1}_{x_i+B_{\ell_i}}\r\|_{\lv}^{-\underline{p}}
\int_{\rn}\mathbf{1}_{x_i+B_{\ell_i+\tau}}(x)\lf[\HL\lf(h^{t'}\r)(x)\r]^{1/t'}\,dx\\
&\ls\lf\|\sum_{i=1}^I|\lz_i|^{\underline{p}}
\lf\|\mathbf{1}_{x_i+B_{\ell_i}}\r\|_{\lv}^{-\underline{p}}
\mathbf{1}_{x_i+B_{\ell_i+\tau}}\r\|_{L^{p(\cdot)/\underline{p}}(\rn)}
\lf\|\lf[\HL\lf(h^{t'}\r)\r]^{1/t'}\r\|_{L^{(p(\cdot)/\underline{p})'}(\rn)}.
\end{align*}
On another hand, by the fact that $p_+/\underline{p}\in(0,t)$, we know that $(p(\cdot)/\underline{p})'\in(t',\fz]$. From this, Lemma \ref{3l1},
\cite[Lemma 3.3(ii)]{lyy17hl}, Lemma \ref{3l5}, the fact that
$\|h\|_{L^{(p(\cdot)/\underline{p})'}(\rn)}\le 1$ and \eqref{4e5},
we further deduce that
\begin{align*}
{\rm I}_1
&\ls\lf\|\sum_{i=1}^I|\lz_i|^{\underline{p}}
\lf\|\mathbf{1}_{x_i+B_{\ell_i}}\r\|_{\lv}^{-\underline{p}}
\mathbf{1}_{x_i+B_{\ell_i}}\r\|_{L^{p(\cdot)/\underline{p}}(\rn)}^{1/\underline{p}}
\|h\|_{L^{(p(\cdot)/\underline{p})'}(\rn)}^{1/\underline{p}}\\
&\ls\lf\|\lf\{\sum_{i=1}^I \lf[\frac{|\lz_i|\mathbf{1}_{x_i+B_{\ell_i}}}
{\|\mathbf{1}_{x_i+B_{\ell_i}}\|_{L^{p(\cdot)}(\rn)}}\r]^{\underline{p}}\r\}^
{1/{\underline{p}}}\r\|_{L^{p(\cdot)}(\rn)}
\sim\|f\|_{\vfah}.\noz
\end{align*}

To deal with ${\rm I}_2$, by \cite[(4.13)]{lql}, we find that, for any $i\in[1,I]\cap\nn$
and $x\in (x_i+B_{\ell_i+\tau})^\com$,
\begin{align*}
T(a_i)(x)
\ls\lf\|\mathbf{1}_{x_i+B_{\ell_i}}\r\|_{\lv}^{-1}
\lf[\HL\lf(\mathbf{1}_{x_i+B_{\ell_i}}\r)(x)\r]^\delta,
\end{align*}
where
\begin{align*}
\delta:=\lf(\frac{\ln b}{\ln \lambda_-}+s_0+1\r)
\frac{\ln \lambda_-}{\ln b}>\frac1{\underline{p}}.
\end{align*}
From this and Lemmas \ref{3l5} and \ref{3l6}, it follows that
\begin{align*}
\|{\rm I}_2\|_{\lv}
&\ls\lf\|\sum_{i\in\nn}\frac{|\lz_i|}{\|\mathbf{1}_{x_i+B_{\ell_i}}\|_{\lv}}
\lf[\HL\lf(\mathbf{1}_{x_i+B_{\ell_i}}\r)\r]^\delta\r\|_{\lv}\\
&\sim\lf\|\lf\{\sum_{i\in\nn}\frac{|\lz_i|}{\|\mathbf{1}_{x_i+B_{\ell_i}}\|_{\lv}}
\lf[\HL(\mathbf{1}_{x_i+B_{\ell_i}})\r]^\delta\r\}^
{1/\delta}\r\|_{L^{\delta p(\cdot)}}^{\delta}\\
&\ls\lf\|\sum_{i\in\nn}\frac{|\lz_i|\mathbf{1}_{x_i+B_{\ell_i}}}
{\|\mathbf{1}_{x_i+B_{\ell_i}}\|_{\lv}}\r\|_{\lv}
\ls\lf\|\lf\{\sum_{i\in\nn}\lf[\frac{|\lz_i|\mathbf{1}_{x_i+B_{\ell_i}}}
{\|\mathbf{1}_{x_i+B_{\ell_i}}\|_{\lv}}\r]^
{\underline{p}}\r\}^{1/\underline{p}}\r\|_{\lv}\\
&\sim\|f\|_{\vfah}.
\end{align*}
This, combined with \eqref{4e7} and Lemma \ref{4l1}(i), implies that
\eqref{4e6} holds true for any $f\in H_{A,\,{\rm fin}}^{p(\cdot),r,s}(\rn)$,
which completes the proof of Step 1).

\emph{Step 2).} By the obvious density of $\vfah$ in $\vh$ with respect to the
quasi-norm $\|\cdot\|_{\vh}$ and a proof similar to the estimation of \eqref{4e3},
we conclude that, for any $f\in\vh$, \eqref{4e6}
also holds true. This finishes the proof of Theorem \ref{4t3}.
\end{proof}

\begin{remark}
\begin{enumerate}
\item[{\rm (i)}] Let $\nu\in\nn$ and $p\in(0,1]$ satisfy
\begin{align}\label{4e8}
\frac1p-1\le\frac{(\ln\lz_-)^2}{\ln b\ln\lz_+}\nu.
\end{align}
If $p(\cdot)\equiv p$, then the spaces $\vh$ and $\lv$ become, respectively, the
anisotropic Hardy space $H^p_A(\rn)$ of Bownik \cite{mb03} and the Lebesgue space
$L^p(\rn)$. In this case, by Theorems \ref{4t2} and \ref{4t3},
we know that, for any $\nu\in\nn$ and $p\in(0,1]$ as in \eqref{4e8}, the anisotropic
Calder\'on--Zygmund operator of order $\nu$ (see Definition \ref{4d2}) is
bounded from $H^p_A(\rn)$ to itself [or to $L^p(\rn)$],
which are just, respectively, \cite[p.\,68, Theorem 9.8 and p.\,69, Theorem 9.9]{mb03}.
Moreover, let
$A:=d\,{\rm I}_{n\times n}$ for some $d\in\rr$ with $|d|\in(1,\fz)$, $\nu=1$. Then
$\frac{(\ln\lz_-)^2}{\ln b\ln\lz_+}\nu=\frac1n$ and $\vh$ and $\lv$ become the classical
isotropic Hardy space $H^p(\rn)$ and the Lebesgue space $L^p(\rn)$, respectively.
In this case, by Theorems \ref{4t2} and \ref{4t3} and Remark \ref{4r1}, we
further conclude that, for any $p\in(\frac n{n+1},1]$, the classical Calder\'on--Zygmund
operator is bounded from $H^p(\rn)$ to itself [or to $L^p(\rn)$], which is a well-known
result (see, for instance, \cite{s93}).

\item[{\rm (ii)}] When $A:=d\,{\rm I}_{n\times n}$ for some $d\in\rr$ with $|d|\in(1,\fz)$,
the space $\vh$ becomes the variable Hardy space $H^{p(\cdot)}(\rn)$ (see \cite{cw14,ns12}).
In this case, Theorems \ref{4t2} and \ref{4t3} are new.

\item[{\rm (iii)}] Very recently, Bownik et al. \cite{bll} obtained the boundedness of a
kind of more general anisotropic Calder\'on--Zygmund operators
(see \cite[Definition 5.4]{bll}) from the anisotropic Hardy space $H^p(\Theta)$ to itself
or to the Lebesgue space $L^p(\rn)$ (see, respectively, \cite[Theorems 5.11 and 5.12]{bll}),
where $\Theta$ is a continuous multi-level ellipsoid cover of $\rn$
(see \cite[Definition 2.1]{bll}). Obviously, the space $\vh$, in this article, is not
covered by the space $H^p(\Theta)$ since the exponent $p$ in $H^p(\Theta)$ is only a
constant. Therefore, Theorems \ref{4t2} and \ref{4t3} are neither covered by
\cite[Theorems 5.11 and 5.12]{bll}.
\end{enumerate}
\end{remark}

\bigskip

\noindent Jun Liu

\medskip

\noindent  School of Mathematics,
China University of Mining and Technology,
Xuzhou 221116, Jiangsu, People's Republic of China

\smallskip

\noindent{\it E-mail:}
\texttt{junliu@cumt.edu.cn} (J. Liu)

\end{document}